\documentclass[12pt]{amsart}
\usepackage{amssymb}
\usepackage{amssymb,amscd,amsmath,amsthm}

\def\QQ{{\mathbb Q}}
\def\ZZ{{\mathbb Z}}
\def\RR{{\mathbb R}}
\def\CC{{\mathbb C}}
\def\BB{{\mathbf B}}

\newtheorem{thm}{Theorem}
\newtheorem{cor}[thm]{Corollary}
\newtheorem{prop}[thm]{Proposition}
\newtheorem{lemma}[thm]{Lemma}
\newtheorem{remark}[thm]{Remark}

\newtheorem{definition}[thm]{Definition}

\begin{document}

\title[Non-homogeneous Fourier matrices]{Classification of non-homogeneous Fourier matrices associated with modular data up to rank $5$}

\author[Gurmail Singh]{Gurmail Singh}
\address{Department of Mathematics and Statistics, University of Regina, Regina, Canada, S4S 0A2}
\email{Gurmail.Singh@uregina.ca}

\date{}

\keywords{Integral modular data, Fusion rings, $C$-algebra, Reality-based algebra.}

\subjclass[2000]{Primary 05E30, Secondary , 05E99, 81R05}

\begin{abstract}
Modular data is an important topic of study in rational conformal field theory, \cite{TG}. A modular datum defines finite dimensional representations of the modular group $\mbox{SL}_2(\ZZ)$. The rows of a Fourier matrix in a modular datum  are scaled to obtain an Allen matrix.
In this paper  we classify the non-homogenous Fourier matrices and non-homogenous Allen matrices up to rank $5$.
Also, we establish some results that are helpful in recognizing the $C$-algebras not arising from Allen matrices.
\end{abstract}

\maketitle

\section{Introduction}

\medskip
A modular datum consists of a pair of Fourier matrix and a diagonal matrix that satisfies the certain properties. The set of columns of a Fourier matrix with componentwise multiplication defined on the columns generate a fusion algebra over $\CC$, see \cite{MC1} and \cite{TG}.
But in this paper we apply a two step scaling on Fourier matrices and obtain the $C$-algebras that have nice properties. At the first step, we normalize the rows of a Fourier matrix $S$ by dividing each row of $S$ by its first entry and obtain an \emph{Allen matrix} $s$. The second scaling is applied on the columns of the Allen matrix $s$ by multiplying  each column of $s$ by its first entry and obtain the matrix $P$. The matrix $P$ is the character table of a $C$-algebra generated by the columns of $P$ under componentwise multiplication, and we call such an algebra a $C$-algebra arising from an Allen matrix $s$. Note that every $C$-algebra has a degree map but every fusion algebra not necessarily has a degree map. Therefore, to classify the Fourier matrices  we use the properties of Allen matrices and $C$-algebras arising from them instead of just using the properties of Fourier matrices or fusion algebras generated by the  columns of Fourier matrices.

\smallskip

Let $A$ be a finite dimensional and commutative algebra over $\CC$ with  distinguished basis  $\mathbf{B} = \{b_0, b_1, \hdots, b_{r-1}\}$, and an $\RR$-linear and $\CC$-conjugate linear involution $*:A \rightarrow A$. Let $\delta: A \rightarrow \CC$ be an algebra homomorphism. Then the triple $(A,\BB,\delta)$ is called a \emph{$C$-algebra} if  satisfies the following properties:

\begin{enumerate}
\item $b_0 =b_0^*= 1_A \in \mathbf{B}$,

\item for all $b_i \in \mathbf{B}$, $(b_i)^* = b_{i^*} \in \mathbf{B}$,
\item the structure constants of $A$ with respect to the basis $\mathbf{B}$ are real numbers, i.e.~for all $b_i, b_j \in \mathbf{B}$, we have
$$b_ib_j = \sum_{b_k \in \mathbf{B}}^{}\lambda_{ijk}b_k, \mbox{ for some } \lambda_{ijk} \in \mathbb{R},$$
\item  for all $b_i, b_j \in \mathbf{B}, \lambda_{ij0} \neq 0 \iff j = i^*$,
\item for all $b_i \in \mathbf{B}, \lambda_{ii^*0} = \lambda_{i^*i0} > 0$.
\item for all $b_i\in \BB$, $\delta(b_i)=\delta(b_{i^*})>0$.
\end{enumerate}

Let $(A,\BB,\delta)$ be a $C$-algebra.  Then $A$ is called \emph{symmetric} if $b_{i^*}=b_i$ for all $i$. The algebra homomorphism $\delta$ is called a \emph{degree map}, and $\delta(b_i)$, for all $b_i\in \BB$, are  called the degrees. For each $i\neq 0$, $\delta(b_i)$ is called a nontrivial degree and $\delta(b_0)=1$ is called a trivial degree. If  $\delta(b_i)=\lambda_{ii^*0}$ for all $b_i\in \BB$ then we call the basis $\BB$ standard basis. Every $C$-algebra with positive degree map has a unique standard basis, which can be arranged by rescaling each basis element by $\delta(b_i)/\lambda_{ii^*0}$. The algebra $A$ is a $C^*$-algebra whose involution satisfies $\alpha^* = \sum_i \bar{\alpha}b_{i^*}$, for all $\alpha = \sum_i \alpha_i b_i \in A$.   In particular, $A$ is a $r$-dimensional semisimple algebra, see \cite{HSnew1}, \cite{G1}.

Group algebras of finite abelian groups, Bose-Mesner algebras of finite commutative association schemes, and commutative fusion algebras are all special types of $C$-algebras.  For the detailed discussion on $C$-algebras and its non-commutative generalizations, see \cite{AFM}, \cite{Hig87},  \cite{HSnew1}, \cite{G1}, and references therein.

\smallskip

Let $(A,\mathbf{B}, \delta)$ be $C$-algebra. Then, we define its {\it order} to be
$$ n := \delta(\mathbf{B}^+) = \sum_{i=0}^d \delta(b_i), $$
and its {\it standard feasible trace} to be $\rho: A \rightarrow \mathbb{C}$ with $\rho(\alpha) = n \alpha_0$ for all $\alpha = \sum_i \alpha_i b_i \in A$.  Note that $\rho$ really is a trace function on $A$ because $\rho(\alpha\beta)=\rho(\beta\alpha)$ for all $\alpha, \beta \in A$.    Positivity of the degree map makes $A$ into a Frobenius algebra with nondegenerate hermitian form
$$ \langle \alpha, \beta \rangle = \rho(\alpha \beta^*), \quad \mbox{ for all } \alpha,\beta \in A. $$
Being a nonsingular trace function on the finite-dimensional Frobenius algebra $A$, $\rho$ can be expressed as a  linear combination of the irreducible characters of $A$, see \cite{Hig87}.  Let $Irr(A)$ denote the set of irreducible characters of $A$. The coefficients $m_{\chi}$ in this linear combination $\rho = \sum_{\chi} m_{\chi} \chi$ are the {\it multiplicities} of a $C$-algebra, where sum runs over $\chi \in Irr(A)$.
The following discussion shows that the multiplicities are always positive real numbers, also see \cite{HIB2} and \cite{HIB1}.
The standard feasible trace is a character, known as \emph{standard character} precisely when all the multiplicities are positive rational integers, see \cite{JA}, \cite{HIB2} and  \cite{HS2}. The standard feasible trace is a \emph{pseudo-standard character} if all the multiplicities are not integers but rational numbers, see \cite{HSnew2}.

\smallskip

Let $(A,\mathbf{B},\delta)$ be a $C$-algebra with order $n$.  Higman's character formula
$$ e_{\chi} = \frac{m_{\chi}}{n} \sum_i \frac{\chi(b_i^*)}{\lambda_{ii^*0}}b_i $$
expresses the centrally primitive idempotents of $A$ in terms of the standardized basis $\mathbf{B}$, see  \cite{Hig87}.
Since $\chi(b_{i^*})=\overline{\chi(b_i)}$ for all $b_i \in \mathbf{B}$, we have
$$\chi(e_{\chi}) = \chi(b_0) = \frac{m_{\chi}}{n}\sum_i \frac{\overline{\chi(b_i)}\chi(b_i)}{\lambda_{ii^*0}}.$$
$\chi(b_0)$ is the degree of $\chi$ so it is always a positive integer.  When the degrees $\delta(b_i)$ are all positive and real, then $n$ is positive and real, so the multiplicity $m_{\chi}$ is positive and real, see \cite{HIB2} and \cite{HIB1}.   It follows then that $e_{\chi}^* = e_{\chi}$.
From the above expression for central idempotents $e_\chi$ we obtain the   orthogonality relation \cite{Hig87}
$$\sum\limits_{k=0}^d \frac{\chi_i(b_{k^*})}{\lambda_{kk^*0}}\chi_j(b_k)=\delta_{ij}\frac {\chi_i(b_0)\delta(\BB^+)}{m_{\chi_i}}.$$
Though we remark that the above orthogonality relation is also true for the noncommutative generalizations of $C$-algebras, see \cite{AFM},  \cite{HSnew1}, \cite{Hig87} and \cite{G1}.

\medskip

In Section 2, we give the definition of  Allen matrices and construct the $C$-algebras from them.
In Section 3, we summarize the results that are useful to recognize the $C$-algebras not arising from Allen matrices. In Section 4, first we prove an interesting theorem that is useful for the remaining results. Then we classify the non-homogenous Fourier matrices and  non-homogenous Allen matrices. For the classification of the non-homogenous Fourier matrices and non-homogenous Allen matrices of rank $4$ and $5$  we assume that one of the nontrivial degrees is equal to $1$. It is not hard recover the Allen matrix and Fourier matrix from the character table, the matrix $P$, of a $C$-algebras arising from an Allen matrix. Therefore, in section 4 we give the character tables instead of Allen matrices and Fourier matrices. In Section 5, we classify the integral Fourier matrices of rank $4$ and $5$ without assuming any nontrivial degree equal to $1$. 

\section{Construction of $C$-algebras arising from Allen matrices}

In this section, we give definitions of the Allen matrices,  Fourier matrices and construct $C$-algebras from them. The Fourier matrices are commonly studied in investigations of modular data, see \cite{MC} and  \cite{MC1}.
Modular data and quasi-modular data are defined differently in the literature. To keep the generality, we assume the structure constants to be integers instead of just nonnegative integers, see \cite{BI}, \cite{MC}, \cite{MC1} and \cite{TG}.

\begin{definition}\label{ModularDef} Let $r\in \ZZ^+$ and  $I$ be an $r \times r$ identity matrix. A pair $(S,T)$ of $r\times r$ complex matrices is called modular datum  (quasi-modular datum, respectively) if
\begin{enumerate}
\item $S$ is a unitary and symmetric matrix, so $S\bar{S}^T = 1, S= S^T$,
\item  $T$ is diagonal matrix and of finite order,
\item $S_{i0}>0$ for $0\leq i\leq r-1$,
\item $(ST)^3=S^2$ $\big($$(ST)^3=I$, respectively$\big)$,
\item $N_{ijk}= \sum_l{S_{li}S_{lj}\bar S_{lk}}{S^{-1}_{l0}} \in \ZZ$, for all $0\leq i,j,k\leq r-1$.
\end{enumerate}
 \end{definition}

\begin{definition}
A matrix $S$ satisfying the axioms $(i), (iii)$ and $(v)$ of Definition \ref{ModularDef} is called a \emph{Fourier matrix}.
\end{definition}

\begin{definition} Let $S$ be a Fourier matrix. We call a matrix $s =[s_{ij}]$ an \emph{Allen matrix} if $s_{ij}= {S_{ij}}{S^{-1}_{i0}}$ for all $i,j$.
\end{definition}

Since the structure constants are algebraic integers, the entries of an Allen matrix are algebraic integers, see \cite[Proposition 2.17]{HIB2} and \cite{MC1}. Therefore, in fact, an Allen matrix with rational entries has rational integer entries. Since $s_{ij}= {S_{ij}}{S^{-1}_{i0}}$ for all $i,j$, $S \in \RR^{r\times r}$ if and only if $s \in \RR^{r\times r}$. Thus we have the following definition.

\begin{definition}
Let $S$ be a Fourier matrix and $s$ be the corresponding Allen matrix. If $S \in \RR^{r\times r}$ $(S \in \QQ^{r\times r},$ resp.$)$ then we call the corresponding Allen matrix $s$ a real (integral, resp.) Allen matrix.
\end{definition}

Note that for an Allen matrix $s$, $s_{i0} = 1$ for all $0\leq i \leq r-1$ and  $s\bar s^T$ is a diagonal matrix. If $s \bar s^T=$diag$(d_0,d_1,\hdots, d_{r-1})$, where $d_i=\sum_js_{ij}\bar s_{ij}$ then  $N_{ijk}= \sum_l{s_{li}s_{lj}\bar s_{lk}}{d^{-1}_l}$ for all $0\leq i,j,k\leq r-1$. Also,  $d_i|s_{ji}|^2 = d_j|s_{ij}|^2$,  for all $0\leq i,j\leq r-1$.
If an Allen matrix $s$ and an associated diagonal matrix $T$ are integral matrices then the modular datum $(s,T)$ is called an \emph{integral modular datum} and the Allen matrix $s$ is called an \emph{integral Fourier matrix}. The study of integral modular datum is of independent interest. The following definitions of an integral modular datum and an integral Fourier matrix are from Cuntz's paper \cite{MC1}.

\begin{definition}\label{IntModularDef} Let $r\in \ZZ^+$. A pair $(s,T)$ of $r\times r$ integral matrices is called an integral modular datum if
\begin{enumerate}
\item $s_{i0} = 1$ for $0\leq i\leq r-1$, det$(s)\neq 0$,
\item  $ss^T=$diag$(d_0,d_1, \hdots, d_{r-1})$, where $d_i=\sum_js^2_{ij}$,
\item $\sqrt{d_j}s_{ij}=\sqrt{d_i}s_{ji}$ for $0\leq i,j\leq r-1$, so $s$ is symmetrizable,
\item $N_{ijk}= \sum_l {s_{li}s_{lj}s_{lk}}{d^{-1}_l} \in \ZZ$, for all $0\leq i,j,k\leq r-1$,
\item $T$ is a diagonal matrix of finite order,
\item $S^2 = (ST)^3$, where $S_{ij}=s_{ij}/\sqrt{d_i}$.
\end{enumerate}
 \end{definition}

\begin{definition}
A matrix $s$ satisfying the axioms $(i), (iii)$ and $(v)$ of Definition \ref{IntModularDef}  is called an \emph{integral Fourier matrix}.
\end{definition}

Therefore, in fact,  an integral Fourier matrix is an integral Allen matrix. Hence, every result that is true for Allen matrices is also true for integral Fourier matrices.
Furthermore, the real Allen matrices of a specific class become integral Fourier matrices, see \cite[Theorem 25]{G1}. But an Allen matrix need not be an integral Fourier matrix, for example, the character table of cyclic group of order $4$ is an Allen matrix but not an integral Fourier matrix.

\begin{definition}
Let $s$  be an Allen matrix (integral Fourier matrix) and $s \bar s^T=$diag$(d_0,d_1,\hdots, d_{r-1})$, where $d_i=\sum_js_{ij}\bar s_{ij}$ for all $0\leq i,j\leq r-1$. The magnitudes of the row vectors of $s$, $d_0,d_1,\hdots,d_{r-1}$ are called norms. We call $d_0$ the principal norm and $d_1,d_2,\hdots,d_{r-1}$  non-principal norms. The principal norm $d_0$ is also known as the \emph{size of a modular datum}.
\end{definition}


There is an interesting row-and-column operation procedure that can be applied to a Fourier matrix $S$ that results in a first eigenmatrix $P$, character table, of a self-dual $C$-algebra.
At the first step, we scale the rows of a Fourier matrix $S$  by dividing the each row by its first entry to get an Allen matrix $s$ with entries $s_{ij}=S_{ij}S^{-1}_{i0}$ for all $0\leq i,j\leq r-1$. At the second step, we further scale the columns of the Allen matrix $s$ by multiplying each column of $s$ by its first entry to  get the matrix $P$ with entries $p_{ij}=s_{ij}s_{0j}$ for all $0\leq i,j\leq r-1$. The matrix $P$ obtained in such a way is a character table, also know as first eigenmatrix, of the self-dual $C$-algebra with standard basis consists of the columns of $P$. The self-duality property of the $C$-algebra comes from the symmetry of the Fourier matrix $S$, see Theorem \ref{CAlgebraThm}.
We can reverse the process to get the Fourier matrix $S$ from the first eigenmatrix $P$. The Allen matrix $s$ is retrieved from the matrix $P$, the first eigenmatrix, by dividing each column of $P$ by the square root of its first entry: $s_{ij} = \dfrac{p_{ij}}{\sqrt{p_{0j}}}$. Further, to retrieve the Fourier matrix $S$ we divide the each row of the Allen matrix $s$ by the square root  its norm: $S_{ij}=\dfrac{s_{ij}}{\sqrt{d_i}}$, where $d_{i}=\sum_js_{ij}\bar s_{ij}$. (For the detailed discussion on the correspondence of matrices $S$, $s$ and $P$ see \cite{G1}.) Thus it is not hard to obtain either one of three matrices $S,s$ and $P$ from the remaining two matrices.
Therefore, it is enough to classify either one of these three matrices. In Section 4, we shall give the first eigenmatrices of rank $2$ to $5$ because we can recover both Allen matrices and Fourier matrices from them.


\begin{remark}\rm Throughout of this paper we use the following notations. Let $S(=[S_{ij}])$ be a Fourier matrix. Then  $s(=[s_{ij}])$ and $P(=[P_{ij}])$ denote the Allen matrix and the first eigenmatrix, respectively. We label the rows and columns with $0,1,2, \dots$.
The set of the columns of matrices $P$ and $s$ are denoted by $\BB=\{b_0,b_1,\hdots, b_{r-1}\}$ and $\tilde \BB=\{\tilde b_0, \tilde b_1,\hdots, \tilde b_{r-1}\}$, respectively. The structure constants generated by the columns, with the componentwise multiplication, of $P$ and $s$ are denoted by $\lambda_{ijk}$ and $N_{ijk}$, respectively. For any matrix $M$  the transpose of $M$ is denoted by $M^T$. A primitive $n$th root of unity is denoted by $\zeta_n$.
\end{remark}

For the remaining part of this section we collect some results and definitions about $C$-algebras arising from Allen matrices that are used frequently in this paper. Although the proofs of these results can be found in the author's another paper  \cite{G1}  but to show that how the construction of $C$-algebra from Allen matrices works we include the proof of two result here.

\smallskip

\begin{lemma}\cite{G1} \label{CAlgebraThm}Let $S$ be a Fourier matrix. Let $\BB$ and $\tilde\BB$ be the sets of the columns of matrices $P$ and $s$, respectively. Let $\CC\BB$ and $\CC\tilde\BB$ be the linear spans of $\BB$ and $\tilde\BB$ over the field of  complex numbers. Prove that the vector spaces $\CC\BB$ and  $\CC\tilde\BB$ satisfy the first five conditions of the definition of  $C$-algebra.
\end{lemma}

\begin{proof}
The $\CC$-conjugate linear involution $*$ on columns of $S$ is given by the involution on elements of $S$,  defined as $(S_{ij})^*=S_{ij^*}=\bar S_{ij}$ for all $i,j$. Thus, if $S_j$ denotes the $j$th column of a Fourier matrix $S$ then the involution on $S_j$ is defined as $(S_j)^*=S_{j^*}=[S_{0j^*}, S_{1j^*}, \hdots, S_{(r-1),j^*}]^T$. Since $b_i=s_{0i}\tilde b_{i}$, the structure constants generated by the basis $\BB$ are given by $\lambda_{ijk}={N_{ijk}s_{0i}s_{0j}}{s^{-1}_{0k}}$, for all $i,j, k$.  As $S$ is a unitary and symmetric matrix, therefore, $N_{ij0}= \sum_l{S_{li}S_{lj}\bar S_{l0}}{S^{-1}_{l0}}\neq 0 \iff j=i^*$ and $N_{ii^*0}=1>0$ for all $i,j$. Hence, $\lambda_{ij0}\neq 0 \iff j=i^*$  and $\lambda_{ii^*0}>0$ for all $i,j$. Also, the first column each of the matrices $P$ and $s$ is an identity element of $\CC\BB$ and $\CC\tilde\BB$, respectively. Therefore, the vector spaces $\CC\BB$ and $\CC\tilde\BB$ satisfy the first five axioms of the definition of a $C$-algebra.
\end{proof}


In the next theorem we show that corresponding to every Allen matrix there exists a $C$-algebra.

\begin{thm}\cite{G1} \label{CAlgebraThm}Let $S$ be a Fourier matrix. Let $\BB$ the set of the columns of matrices $P$. Prove that the vector space $\CC\BB$ is a $C$-algebra and $\BB$ is the standard basis.
\end{thm}

\begin{proof}We just need to check the last axiom of the definition of a $C$-algebra. Let $\tilde \BB$ and $\BB$ denote the columns of $s$ and $P$ respectively. Define a $\CC$-conjugate linear map $\delta:A\longrightarrow \CC$  as  $\delta(\sum_ia_i\tilde{b_i})=\sum_i\bar a_is_{0i}$. Thus $\delta(b_i)= \delta(s_{0i}\tilde b_{i}) = s^2_{0i}$, hence $\delta$ is positive valued. Since $b_{i^*}= \bar s_{0i}\tilde{b}_{i^*}=s_{0i}\tilde{b}_{i^*}$, $(b_ib_{i^*})_0=s_{0i}s_{0i}(\tilde b_i \tilde b_{i^*})_0= s^2_{0i}=\delta(b_i)$. The map $\delta$ is an algebra homomorphism, to see:
$$\begin{array}{rcl}\delta(b_i b_j)&=&
s_{0i}s_{0j}\sum\limits^{}_{k}N_{ijk}\delta(\tilde{b_k})
\\&=&s_{0i}s_{0j}[\dfrac{s_{0i}s_{0j}}{d_0} \sum\limits^{r-1}_{k=0}\bar s_{0k}s_{0k}]+s_{0i}s_{0j}[\sum\limits^{r-1}_{l=1}\dfrac{s_{li}s_{lj}}{d_l}
\sum\limits^{r-1}_{k=0}\bar s_{lk}s_{0k}]
\\&=&s_{0i}s_{0j}[\dfrac{s_{0i}s_{0j}}{d_0} d_0]+s_{0i}s_{0j}[0]
\\&=&s^2_{0i}s^2_{0j}=\delta(b_i)\delta(b_j)\end{array}$$
In the third last equality we use the fact that $ \sum\limits^{n}_{k=1} s_{0k}\bar  s_{0k} =d_0$ and the rows of $s$ are orthogonal.
\end{proof}

Let $\tilde \BB$ and $\BB$ denote the columns of $s$ and $P$, respectively. The both vector spaces $\CC\tilde \BB$ and $\CC\BB$ are fusion algebras.
By the above theorem, $\CC\BB$ is also a $C$-algebra. In this paper we make use of the interrelationship between the norms of an Allen matrix and degrees of an associated $C$-algebra.

\begin{definition}\cite{G1} Let $\BB=\{b_0,b_1,\hdots, b_{r-1}\}$ denote the columns of a matrix $P$ obtained from an Allen matrix (integral Fourier matrix, respectively) $s$. Then the $C$-algebra $(A,\BB, \delta)$ with basis $\BB$ is called a $C$-algebra arising from an Allen matrix (integral Fourier matrix, respectively) $s$.
\end{definition}

\begin{definition}\cite{G1}
Let $(A,\BB, \delta)$ be a $C$-algebra arising from an Allen matrix $s$ with standard basis $\BB=\{b_0,b_1,\hdots, b_{r-1}\}$. Let $k\in \ZZ^+$ and $\delta(b_i)=k$ for all $1\leq i\leq r-1$. Then $A$ is called a \emph{homogenous $C$-algebra} with \emph{homogeneity degree} $k$, and the associated Fourier matrix $S$ (Allen matrix $s$, respectively) is called a \emph{homogeneous Fourier matrix} (\emph{homogeneous Allen matrix}, respectively). If a Fourier matrix $S$ (Allen matrix $s$, respectively) is not a homogenous matrix then it is called a \emph{non-homogeneous Fourier matrix $S$} (\emph{non-homogeneous Allen matrix $s$}, respectively).
\end{definition}

\begin{lemma}\cite{G1} \label{SqaureDegreesForIntegralFurierCaseLemma}Let $(A,\BB)$ be a $C$-algebra arising from an integral Fourier matrix $s$ with standard  basis $\BB=\{b_0,b_1,\hdots, b_{r-1}\}$. If $\delta(b_i)=k_i$ for all $i$ then $k_i$ are square integers.
\end{lemma}

\begin{prop}\cite{G1} \label{IntegralNormsAndDegreesProp}Let $(A,\BB)$ be a $C$-algebra arising from an Allen matrix $s$. Then the norms and the degrees of $A$ are rational integers.
\end{prop}

The following proposition states that the degrees (multiplicities) of a $C$-algebra  arising from an Allen matrix divide the order of the $C$-algebra. In addition, it states that not only the list of multiplicities matches with the list of degrees but also their indices match.

\begin{prop}\cite{G1}\label{SymmetrizingPropertyProp}
Let $(A,\mathbf{B}, \delta)$ be a $C$-algebra arising from Allen matrix $s$.
\begin{enumerate}
 \item  The multiplicities (degrees) divide the order of $A$, Also, $m_j= d_0/d_j$ for all $j$.

  \item The degrees of $A$ exactly match with the multiplicities of $A$, that is, $m_j=\delta(b_j)$ for all $j$.
\end{enumerate}
\end{prop}

%
%


\begin{definition}\cite{G1}\label{AllenIntDef} Let $(A,\BB, \delta)$ be a $C$-algebra with standard basis $\BB=\{b_0,b_1,\hdots, b_{r-1}\}$ with a character table $P$. We say that $A$  satisfy the \emph{Allen integrality condition} if ${\lambda_{ijk}\sqrt{\delta(b_k)}}/{\sqrt{\delta(b_i)\delta(b_j)}}\in \ZZ$ for all $i,j,k$.
\end{definition}


If a $C$-algebra is arising from an Allen matrix then it is self-dual. In general, the converse is not true. In the next theorem we states that the converse is also true for $C$-algebras satisfying the Allen integrality condition.

\begin{thm}\cite{G1}\label{IntegralCAlgebrasAndFourierMatrices}
Let $(A,\BB, \delta)$ be a $C$-algebra with standard basis $\BB=\{b_0,b_1, \hdots, b_{r-1}\}$. Then $A$ is self-dual and satisfies the Allen integrality condition if and only if $A$ arises from an Allen matrix.
\end{thm}

\smallskip
%


\section{General results on non-homogeneous $C$-algebras arising from Allen matrices}

\medskip

In this section, we prove the results that are helpful in recognizing the $C$-algebras that are not arising from Allen matrices by mere looking at the character tables of the $C$-algebras, in particular, the first row of the character tables. Note that an adjacency algebra of a commutative association scheme is a $C$-algebra. All the character tables of the association schemes used here are produced by Hanaki and Miyamoto, see \cite{HM}. In this section, after the proof of each proposition we will give only few examples of the character tables of the association schemes just to demonstrate the application of the proposition. Though we remark that the character tables given in \cite{HM} have information about the multiplicities but we do not need that information  to apply the following results.

\smallskip

In the following lemma we prove that it does not require all the structure constants to be nonnegative for the subset of basis elements with degree $1$ to become a group.

\begin{lemma}\label{ElementaryGroupLemma}
Let $(A,\BB, \delta)$ be a $C$-algebra arising from an  Allen matrix $s$ with standard basis $\BB=\{b_0,b_1,\hdots, b_{r-1}\}$. Let the structure constants $\lambda_{ii^*j}\geq 0$ for all $i,j$. If $L=\{b\in \BB: \delta(b)=1\}$ then $L$ is an abelian group.
\end{lemma}

\begin{proof}
As $\delta(b_0)=1$, $b_0\in L$. Also $b_ib_{i^*}=\lambda_{ii^*0}b_0+\sum_j\lambda_{ii^*j}b_j$ for all $i$. Since $\lambda_{ii^*j}\geq 0$ and  $\delta(b_i)=1$,  $b_ib_{i^*}=b_0$ for all $i$.  Thus, $b^{-1}_i=b_{i^*}$.  Therefore, $L$ is a group. Since $A$ is commutative, $L$ is an  abelian group.
\end{proof}

In the next theorem, we collect the conditions under which a $C$-algebra arising from a real Allen matrix cannot have just one degree different from $1$.

\begin{thm}\label{OneDegreeDifferentThm} Let $(A,\BB, \delta)$ be a $C$-algebra arising from a real Allen matrix $s$ with standard basis $\BB=\{b_0,b_1,\hdots, b_{r-1}\}$. Let the structure constants $\lambda_{iij}\geq 0$ for all $i,j$.
\begin{enumerate}
  \item If the rank of $A$ is an even integer but greater than $2$, then $A$ cannot have only one degree different from $1$.
  \item If rank of $A$ is greater than $3$, and  $|p_{ij}|\leq p_{0j}$ for all $i,j$. Then $A$ cannot have only one degree greater or equal to $r$ and the remaining degrees equal to $1$.
\end{enumerate}
\end{thm}

\begin{proof}  For $(i)$. By Lemma \ref{ElementaryGroupLemma}, the elements of the basis $\BB$ with degree $1$ form an elementary abelian group, the order of the group is an even integer. The assertion follows from the order considerations.

\smallskip

For $(ii)$. Let only one degree of $A$ be greater or equal to $r$ and the remaining degrees be equal to $1$. Without loss of generality, let first row of the character table  be  $[1,1, \hdots,1,k]$, where $k\geq r$, see Theorem \ref{SimulataneousPermThm}. 
Since $m_0=m_1=\hdots=m_{r-2}$, $d_0=d_1=\hdots=d_{r-2}$. As $|p_{ij}|\leq p_{0j}$, therefore the only possible entries of row 2, row 3, $\hdots$, row $r-2$ of $P$ are  $[1,1, \hdots,1,-k]$, which is not possible as  $P$  is nonsingular.
\end{proof}

The Part $(i)$ of the above theorem is also true if the order of the algebra is greater than $2$, see Theorem \ref{RankTwoThm}. The character table of the association scheme {\tt as4(2)} is an example where Part $(ii)$ of the above theorem fails for the rank $3$.
The examples of the $C$-algebras, but not self-dual, include the adjacency algebras of the association schemes {\tt as6(2), as6(4), as8(2)} and {\tt as8(8)}.


\begin{cor}\label{OneDegreeDifferentCor} Let $(A,\BB, \delta)$ be a $C$-algebra arising from a real Allen matrix $s$ with standard basis $\BB=\{b_0,b_1,\hdots, b_{r-1}\}$. If $2$ does not divide the order of $A$ then $\delta(b_i)\neq 1$ for some $i>0$.
\end{cor}

%
%
%

The next proposition helps to recognize the $C$-algebras not arising from Allen matrices.

\begin{prop}Let $(A,\BB, \delta)$ be a $C$-algebra arising from an Allen matrix with standard  basis $\BB=\{b_0,b_1,\hdots, b_{r-1}\}$. Let $\delta(b_i)=k_i$ for all $i$.  Let the number of $i$'s such that  $\delta(b_i)=1$ be $t$ and the  remaining $r-t$ degrees $\delta(b_j)=k$. Then the possible values of $k$ are the divisors of $t$.
\end{prop}

\begin{proof}
By  Proposition \ref{IntegralNormsAndDegreesProp}, $\delta(b_l)\in \ZZ$, for all $l$. Therefore, by Proposition \ref{SymmetrizingPropertyProp},  ${d_0}{k^{-1}}\in \ZZ$. Thus $d_0=t+(d_0-t)k$ implies $tk^{-1} \in \ZZ$. Hence $k$ is in the subset of the divisors of $t$.
\end{proof}


For example, from the character table of the association schemes {\tt as9(3), as9(8), as10(6), as16(20), as16(21)} and {\tt as16(62)} etcetera,  we conclude that these are not the character tables of the $C$-algebras, adjacency algebras, arising from Allen matrices. By the above theorem we also conclude that any $C$-algebra of rank $2$ and order greater than $2$ cannot arise from an Allen matrix, also see Theorem \ref{RankTwoThm}.

In the next proposition, we generalize the Cuntz's result about integral Fourier matrices with nonnegative structure constants, see \cite{MC1}. This proposition is for real Allen matrices and it does not require all the structure constants to be nonnegative. 

\begin{prop}
Let $(A,\BB, \delta)$ be a $C$-algebra arising from a real Allen matrix with standard basis $\BB=\{b_0,b_1, \hdots,b_{r-1}\}$. Let  the structure constants $\lambda_{iij}\geq 0$ for all $i,j$. If all the degrees of $A$ different from $1$ are all equal then the nontrivial degree might be power of $2$. Note: the algebra $A$  need not be homogeneous.
\end{prop}

\begin{proof} Let $\delta(b_i) =k$ for some $1\leq i\leq r-1$. By Lemma \ref{ElementaryGroupLemma}, the elements of $\BB$ with degree $1$ form an elementary abelian group. Thus the number of linear elements of $\BB$ is a power of $2$, say $2^m$, $m\in \ZZ^+$. If the number of nonlinear degrees is equal to $n$ then $d_0=2^m+nk$. By Proposition \ref{SymmetrizingPropertyProp}, $k$ divides $d_0$ implies $k$ divides $2^m$.
\end{proof}

By the above proposition,  from the character table of the association schemes {\tt as12(9), as14(4), as16(10), as16(20), as16(21)} and {\tt as16(62)} etcetera, we conclude that these are not the character tables of the $C$-algebras arising from Allen matrices.

In the next proposition we examine the possible number of occurrences of a degree if it is one of the degrees and satisfy certain criteria. 

\begin{prop}\label{PredictionOfDegreesAllenCaseProp}Let $(A,\BB, \delta)$ be a $C$-algebra arising from an Allen matrix, with standard  basis $\BB=\{b_0,b_1,\hdots, b_{r-1}\}$. Let $\delta(b_i)=k_i$ for all $i\geq 1$.  If for a given $j$, $k_j(\neq 1)$ is a smallest nontrivial degree that divides all the nontrivial degrees $k_l(\neq 1)$ then the number of degrees equal to $1$ is a multiple of $k_j$.
In general, let $k_t$ be a degree divisible by all the smaller degrees and divides all the bigger degrees. Let $k_s$ be the largest degree among all the degrees strictly less than $k_t$. Let the sum of the degrees less than $k_s$ be $\beta_1 k_s$ and the number of degrees equal to $k_s$ be $\beta_2$. Then $\beta_1k_s+\beta_2k_s$ is divisible by $k_t$. Note: $\beta_1 k_s$ is not equal to zero only if $k_s>1$.
\end{prop}

\begin{proof}
By Proposition \ref{SymmetrizingPropertyProp},  ${d_0}{\delta(b_i)^{-1}}=d_i$ for each $i$. Therefore, $d_0=1+\sum\limits^{r-1}_{i=1}k_i$ and ${d_0}{k^{-1}_j} \in \ZZ$ implies $ ({1+\sum\limits^{r-1}_{i=1}k_i}){k^{-1}_j}
=({\sum\limits^{}_{k_i=1}k_i+\sum\limits^{}_{k_i\geq k_j}k_i}){k^{-1}_j}
=({\sum\limits^{}_{k_i=1}k_i}){k^{-1}_j}+ \alpha
 \in \ZZ$, where $\alpha\in \ZZ$. Thus the number of degrees equal to $1$ are multiple of $k_j$.

Similarly,  ${d_0}{k^{-1}_t} \in \ZZ$ implies $ ({\sum\limits^{r-1}_{i=0}k_i}){k^{-1}_t}
=({\sum\limits^{}_{k_i<k_s}k_i+\sum\limits^{}_{k_i=k_s}k_i
+\sum\limits^{}_{k_i> k_s}k_i}){k^{-1}_t}
=(\beta_1 k_s +\sum\limits^{}_{k_i=k_s}k_i){k^{-1}_t}+ \gamma
 \in \ZZ$, where $\gamma\in \ZZ$. Thus $\beta_1k_s+\beta_2k_s$ is divisible by $k_t$.
\end{proof}

For example, the character tables of association schemes {\tt as7(2), as8(5), as8(6), as9(8), as9(9)} etcetera, all the character tables of adjacency algebras of association schemes of rank $2$ and order greater than $2$, and the character tables of all adjacency algebras of homogenous schemes violate the conditions of the above proposition so they are not character tables of the association schemes  arising from Allen matrices, for the character tables see \cite{HM}.

\begin{lemma}Let $(A,\BB, \delta)$ be a $C$-algebra arising from an integral Fourier matrix of odd rank and odd order with standard  basis $\BB=\{b_0,b_1,\hdots, b_{r-1}\}$. Let the odd degree among all the degrees of $A$ be maximum. Then the rank of $A$ must be at least $10$.
\end{lemma}

\begin{proof}
Let $d_0=\delta(b_i)a_i$. By \cite[Lemma 3.7]{MC1}, $d_0$ is an odd square, thus $a_i$ is a square.
Let $\delta(b_1)$ be an odd integer and $\delta(b_1)\geq \delta(b_i)$ for each $i$. Therefore, $d_0\geq 9\delta(b_1)$, and $ d_0\leq  1+(r-1)\delta(b_1)$. Thus $ 9\delta(b_1)\leq 1+(r-1)\delta(b_1)$ implies  $\delta(b_1)(9-(r-1))\leq 1$ implies $\delta(b_1) \in \ZZ^+$ only if $r-1\geq9$.
\end{proof}

\section{Non-homogeneous $C$-algebras of rank $2$ to $5$}

\medskip

In this section we classify the non-homogenous Allen matrices and non-homogenous Fourier matrices of rank $2$ to $5$. In fact, we find the first eigenmatrices, the character tables, because the associated Allen matrices and  Fourier matrices can be recovered by scaling the columns of first eigenmatrices and then rows of the Allen matrices, respectively, see Section 2. For the classification of homogenous Allen matrices and homogenous Fourier matrices  see \cite{G1}.
All the theorems, propositions and  methods developed in this section can be used and extended for the classification of the Allen matrices and Fourier matrices of higher ranks.

\smallskip

In the next theorem, we prove that any row or column permutation of first eigenmatrix $P$ will lead to the simultaneous row and column permutation of $P$. This nice theorem is useful for the classification of the Allen matrices and Fourier matrices because we do not need to consider all the row and column permutations of a given first eigenmatrix $P$ of a $C$-algebra that we assume is arising from an Allen matrix. 
Note: if a $C$-algebra is arising from an Allen matrix then it is self dual, thus we can always get a character table $P$ of the $C$-algebra such that $P\bar P=nI$, where $n$ is the order of the algebra and $I$ is the identity matrix.

\begin{thm}\label{SimulataneousPermThm} Let $(A,\BB, \delta)$ be a $C$-algebra of rank $r$. Let $n$ be the order of the algebra and $P$ be the character table  such that $P\bar P=nI$.
\begin{enumerate}
  \item Let $Q$ be the matrix obtained from $P$ by simultaneous row  and column permutations. Then $Q\bar Q=nI$.
\item Let $S$ be a symmetric and unitary matrix. Then any simultaneous row and column permutation of $S$ is again a symmetric and unitary matrix.
  \item 
      Let the $C$-algebra be arising from an Allen matrix $s$.
      Let $Q$ be the character table of $A$ not necessarily obtained by scaling the columns of the Allen matrix $s$. If we consider it to be obtained from the Allen matrix $s$ then $Q\bar Q=nI$. That is, any row or column permutation of $P$ will lead to the simultaneous row and column permutation of $P$.
\end{enumerate}
\end{thm}

\begin{proof}
For $(i)$. Let $P=[p_{ij}]$ be the character table of $A$. Since $P\bar P=nI$, for all $0\leq i, j\leq r-1$, $\sum p_{ij}\bar p_{ji}=n$ and $\sum p_{ij}\bar p_{jk}=0$, $k\neq i$. Let $Q=[q_{ij}]$ be the matrix obtained from the matrix $P$ after simultaneous permutation of any number of rows and columns of the matrix $P$.
Therefore, for all $0\leq i, j\leq r-1$, $\sum q_{ij}\bar q_{ji}=\sum p_{ij}\bar p_{ji}=n$ and $\sum q_{ij}\bar q_{jk}=\sum p_{ij}\bar p_{jk}=0$, $k\neq i$. Hence  $Q\bar Q=nI$.

\smallskip

For $(ii)$. Proof is similar to the part $(i)$ above.

\smallskip

For $(iii)$. Here we consider the case of row permutations and the proof for column permutations is similar to the proof for row permutations. Let $P$ be the character table of a $C$-algebra arising from an Allen matrix $s$,  such that $P\bar P=nI$. Let the rows $i_1,i_2, \hdots, i_l$ moves to the positions of the rows $j_1,j_2, \hdots, j_l$. Let $s$ be the Allen matrix obtained from $P$ by scaling the columns with the square roots of the degrees. Thus the row $i_k$ of $s$ moves to the row $j_k$. By Proposition \ref{SymmetrizingPropertyProp}, $d_0m_i=d_i$, where $m_i (=\delta(b_i))$ is the multiplicity of $A$ and this fact will force column $i_k$ to move to column $j_k$. Hence the assertion is true by part $(i)$.
\end{proof}

The next theorem classifies the Allen matrices and Fourier matrices of rank $2$.

\begin{thm}\label{RankTwoThm}
Let $(A,\BB, \delta)$ be a $C$-algebra of rank $2$ arising from an Allen matrix $s$. Then the following are the unique expressions for the Fourier matrix $S$ and the Allen matrix $s$ (integral Fourier matrix).
$$S=\dfrac{1}{\sqrt2}\left[
       \begin{array}{cc}
         1 & 1 \\
         1 & -1 \\
       \end{array}
     \right] \mbox{ and } s=\left[
       \begin{array}{cc}
         1 & 1 \\
         1 & -1 \\
       \end{array}
     \right]$$
\end{thm}

\begin{proof}
Let  $P$ be character table for the $C$-algebra arising from an Allen matrix $s$. Therefore,
$$P=\left[
       \begin{array}{cc}
         1 & k \\
         1 & -1 \\
       \end{array}
     \right],~ s = \left[
       \begin{array}{cc}
         1 & \sqrt{k} \\
         1 & -1/\sqrt{k} \\
       \end{array}
     \right] \mbox{ and } S= \dfrac1{\sqrt{1+k}}\left[
       \begin{array}{cc}
         1 &  \sqrt{k}  \\
          \sqrt{k}  & -1 \\
       \end{array}
     \right].$$
Since $N_{111} =({k-1})/{\sqrt{k}}$ is an integer, $k=1$.
\end{proof}

The above theorem also confirms the decision of the results of the previoous section that conclude that an association scheme of rank $2$ and order greater than $2$ cannot be the association scheme associated to an Allen matrix.

Let $(S,T)$ be a modular datum. Then $(S,\zeta_3T)$ is also a modular datum. Hence, in a modular datum, for a given Fourier matrix there might be at least three different associated torsion  diagonal matrices.
For the Fourier matrix $S=\dfrac{1}{\sqrt2}\left[
       \begin{array}{cc}
         1 & 1 \\
         1 & -1 \\
       \end{array}
     \right]$ and integral Fourier matrix $s=\left[
       \begin{array}{cc}
         1 & 1 \\
         1 & -1 \\
       \end{array}
     \right]$ of the above theorem there exists a torsion matrix $T=$diag$ (\xi ,\mu)$, so that $(S,T)$ and $(s,T)$ are Modular datum and integral Modular datum respectively, see \cite[Example 1]{MC1}. For the complete list of the diagonal matrices associated with the Fourier matrix $S$ see \cite[Example 36]{G1}.


\medskip

Let $P$ be the character table for a symmetric $C$-algebra arising from an Allen matrix $s$ of rank $3$ with standard basis $\BB=\{b_0,b_1,b_2\}$. Let $b_1b_2= ub_1+vb_2$.
Then
$$P=\left[
   \begin{array}{ccc}
     1 & k & l \\
     1 & \phi_1 & \phi_2\\
     1 & \psi_1 & \psi_2 \\
   \end{array}
 \right],
$$ where
$\phi_1= ({v-u-1+\sqrt{(u-v-1)^2+4u}})/{2}$,
$\quad \phi_2= ({u-v-1-\sqrt{(u-v-1)^2+4u}})/{2}$,

$\psi_1=({v-u-1-\sqrt{(u-v-1)^2+4u}})/{2}$ and
$\quad \psi_2= ({u-v-1+\sqrt{(u-v-1)^2+4u}})/{2}$.

\noindent Therefore, $$s=\left[
   \begin{array}{ccc}
     1 & \sqrt{k} & \sqrt{l} \\
     1 & \dfrac{\phi_1}{\sqrt{k}} & \dfrac{\phi_2}{\sqrt{l}}\\
     1 & \dfrac{\psi_1}{\sqrt{k}} & \dfrac{\psi_2}{\sqrt{l}} \\
   \end{array}
 \right].
 $$
Thus
$d_0= 1+k+l=n$, $d_1=1+\dfrac{|\phi_1|^2}{k}+\dfrac{|\phi_2|^2}{l}=\dfrac{n}{k}$,
$d_2=1+\dfrac{|\psi_1|^2}{k}+\dfrac{|\psi_2|^2}{l}=\dfrac{n}{l}$.

\begin{thm}\label{t2}
There is no symmetric homogenous $C$-algebra of rank 3 arising from an Allen matrix.
\end{thm}

\begin{proof} Let $(A,\BB, \delta)$ be a symmetric homogenous $C$-algebra arising from an Allen matrix $s$. Since $S^T=S$, $ \phi_2k= \psi_1l$ implies $u=v$. The structure constant $N_{210}=0$ implies $k=2u$. Thus  $\phi_1=({-1+\sqrt{1+2k}})/{2}$ and $\phi_2=({-1-\sqrt{1+2k}})/{2}$.
Since $A$ is homogeneous, all the degrees are equal to $1$, see \cite[Thorem 25]{G1}.
But the structure constants
$$N_{112}=\dfrac{1}{(2k+1)\sqrt{k}}[k^2+ \dfrac{k}{2}] \mbox{ and } N_{222}
=\dfrac{1}{(2k+1)\sqrt{k}}[k^2-1-\dfrac32k]$$
are not integers for $k=1$.
\end{proof}

\begin{thm}\label{p3}
Let $(A,\BB, \delta)$ be a symmetric $C$-algebra of rank $3$ arising from an Allen matrix $s$. Then the only possible degree pattern of $A$ is $[1,1,2]$ up to permutations and the corresponding matrices $P$, $s$ and $S$ are as follows.
$$P=\left[
   \begin{array}{ccc}
     1 & 1 & 2 \\
     1 & 1 & -2\\
     1 & -1 & 0 \\
   \end{array}
 \right],~~
s=\left[
   \begin{array}{ccc}
     1 & 1 & \sqrt2 \\
     1 & 1 & -\sqrt2\\
     1 & -1 & 0 \\
   \end{array}
 \right] \mbox{ and } S=\left[
   \begin{array}{ccc}
     1/2 & 1/2 & 1/\sqrt2 \\
     1/2 & 1/2 & -1/\sqrt2\\
     1/\sqrt2 & -1/\sqrt2 & 0 \\
   \end{array}
 \right].
$$
\end{thm}

\begin{proof}
Let $1$, $k$ and $l$ be the degrees of $A$. Therefore, by Proposition \ref{SymmetrizingPropertyProp}, $k$ divides $1+l$, and $l$ divides $1+k$. By Theorem \ref{t2}, $A$ cannot be homogenous, thus the only possible  degree pattern of $A$ are $[1,1,2]$ and $[1,2,3]$, up to permutations.
Since $N_{012}=0$, $1-\dfrac{v}{k}-\dfrac{u}{l}=0$. Therefore, the degree patterns $[1,1,2]$ and $[1,2,3]$ imply $v=1-\dfrac{u}{2}$ and $v=2-\dfrac{2u}{3}$, respectively.

\smallskip

Case $1$. Let the degree pattern be $[1,1,2]$, that is, $k=1$ and $l=2$.

Therefore, $v=1-\dfrac{u}2 $ implies
$$\phi_1=\big({-\dfrac{3u}2+\sqrt{(\dfrac{3u}2-2)^2+4u}}\big)/{2} \mbox{ and } \psi_1=({-\dfrac{3u}2-\sqrt{\big(\dfrac{3u}2-2)^2+4u}}\big)/{2}.$$
Since $N_{011}=1$, $u^3(u-1)=0$. Thus $u=0$, or  $u=1$. Hence $(u,v)=(0,1)$, or $(u,v)=(1,1/2)$.

\smallskip

Subcase 1. If $(u,v)=(0,1)$.

Then $\phi_1=1$, $\psi_1=-1$, $\phi_2=-2$ and $\psi_2=0$. Therefore,
$$P=\left[
   \begin{array}{ccc}
     1 & 1 & 2 \\
     1 & 1 & -2\\
     1 & -1 & 0 \\
   \end{array}
 \right],
s=\left[
   \begin{array}{ccc}
     1 & 1 & \sqrt2 \\
     1 & 1 & -\sqrt2\\
     1 & -1 & 0 \\
   \end{array}
 \right] \mbox{ and } S=\left[
   \begin{array}{ccc}
     1/2 & 1/2 & 1/\sqrt2 \\
     1/2 & 1/2 & -1/\sqrt2\\
     1/\sqrt2 & -1/\sqrt2 & 0 \\
   \end{array}
 \right].
$$

Subcase 2. If $(u,v)=(1,\dfrac12)$.

Then $\phi_1=\dfrac{-3+\sqrt{17}}{4}$, $\psi_1=\dfrac{-3-\sqrt{17}}{4}$, $\phi_2=\dfrac{-1-\sqrt{17}}{4}$ and $\psi_2=\dfrac{-1+\sqrt{17}}{4}$.
Therefore,
$$P=\left[
   \begin{array}{ccc}
     1 & 1 & 2 \\
     1 & \dfrac{-3+\sqrt{17}}{4} & \dfrac{-1-\sqrt{17}}{4}\\
     1 & \dfrac{-3-\sqrt{17}}{4} & \dfrac{-1+\sqrt{17}}{4} \\
   \end{array}
 \right]\mbox{ and } s=\left[
   \begin{array}{ccc}
    1 & 1 & \sqrt2 \\
     1 & \dfrac{-3+\sqrt{17}}{4} & \dfrac{-1-\sqrt{17}}{4\sqrt2}\\
     1 & \dfrac{-3-\sqrt{17}}{4} & \dfrac{-1+\sqrt{17}}{4\sqrt2} \\
   \end{array}
 \right].$$
Note that $m_0=m_1$, but $d_0=4$ and $d_1=1+\dfrac{1}{16}[39-3\sqrt{17}]$, which is a contradiction to the Proposition \ref{SymmetrizingPropertyProp}. Hence $u=1$ and $v=\dfrac12$ is not a possible case.


\medskip

Case $2$. Let the degree pattern be $[1,2,3]$, that is, $k=2$ and $l=3$.

\noindent Therefore, $v=2-\dfrac{2u}3 $, implies
$$\phi_1=\big({-\dfrac{5u}3+1+\sqrt{(\dfrac{5u}3-3)^2+4u}}\big)/{2} \mbox{ and } \psi_1=\big({-\dfrac{5u}3+1-\sqrt{(\dfrac{5u}3-3)^2+4u}}\big)/{2}.$$
Since $N_{011}=1$,
$625u^4-1850u^3+2520u^2-1296u+243=0.$ But it has no real roots, see \cite{EL}.
%
%
%
%
%
Thus we rule out $[1,2,3]$ degree pattern, because an Allen matrix associated with a symmetric $C$-algebra might be a real matrix, see \cite[Proposition 21]{G1}.
\end{proof}

\begin{remark}\rm
In fact, the character table $P$ above is the character table of an association scheme {\tt as4(2)}, see \cite{HM}. By Theorem \ref{p3} and  \cite[Proposition 21]{G1}, there is no  symmetric $C$-algebra of rank $3$ arising from an integral Allen matrix (integral Fourier matrix). Thus there is no integral Allen matrix (integral Fourier matrix) of rank $3$.
\end{remark}


In the next theorem we prove that there only one asymmetric $C$-algebra of rank 3 arising from an Allen matrix. Moreover, the following theorem shows that for rank $3$ it is not necessary to assume $|s_{ij}|\leq s_{0j}$ for all $i,j$ to prove that the homogeneous $C$-algebra arising from an Allen matrix is a group algebra, see \cite[Theorem 25]{G1}.

\begin{thm}\label{t9}
Let $(A,\BB, \delta)$ be a asymmetric $C$-algebra arising from an Allen matrix $s$ of rank $3$. Then $s$ is the character table of the cyclic group of order $3$.
\end{thm}

\begin{proof}
The character table $P$ for asymmetric table algebra is as follows, see \cite[\S 4]{HRB}.
$$P=\left[
   \begin{array}{ccc}
     1 & k & k \\
      1 & \alpha & \bar\alpha \\
     1 & \bar\alpha & \alpha\\
    \end{array}
 \right],$$
where $\alpha={(-1+i\sqrt{1+2k})}/{2}.$
By \cite[Theorem 3.3]{HRB}, the structure constants are given by $\lambda_{112} =(k+1)/2$,
$\lambda_{121}= (k-1)/2$. The corresponding $s$ matrix is:
$$s=\left[
   \begin{array}{ccc}
     1 & \sqrt{k} & \sqrt{k} \\
      1 & {\alpha}/{\sqrt{k}} &  {\bar\alpha}/{\sqrt{k}} \\
     1 &  {\bar\alpha}/{\sqrt{k}} &  {\alpha}/{\sqrt{k}}\\
    \end{array}
 \right].
$$
Now $\lambda_{ijk}=  {s^{-1}_{0k}}{s_{0i}s_{0j}N_{ijk}}$ implies $N_{ijk}= {\lambda_{ijk}s_{0k}}{s^{-1}_{0i}s^{-1}_{0j}}$.
Therefore, $N_{112}= ({k+1})/{2\sqrt{k}}$ and
$N_{121}= ({k-1})/{2\sqrt{k}}$. Since $N_{112}$ and $N_{121}$ are positive integers,  $k=1$. Then $\alpha=(-1+i\sqrt{3})/2=\zeta_3$, and
$$s(=P)=\left[
   \begin{array}{ccc}
     1 & 1 & 1 \\
      1 & \zeta_3 & \zeta^2_3 \\
     1 & \zeta^2_3 & \zeta_3\\
    \end{array}
 \right]\mbox{ and }S=\dfrac{1}{\sqrt3}\left[
   \begin{array}{ccc}
     1 & 1 & 1 \\
      1 & \zeta_3 & \zeta^2_3 \\
     1 & \zeta^2_3 & \zeta_3\\
    \end{array}
 \right]
.$$
\end{proof}

%




\smallskip

In the next lemma we show that if one of the nontrivial degree of $C$-algebra arising from an Allen matrix is equal to $1$ then absolute values of elements of the corresponding row of character table of the algebra are completely determined.

\begin{lemma}\label{EqualityOfKandLLemmaRank5}Let $(A,\BB, \delta)$ be a $C$-algebra arising from an Allen matrix $s$ of rank $r$ with standard  basis $\BB=\{b_0,b_1,\hdots, b_{r-1}\}$. Let $P$ be the character table of $A$ and  $|s_{ij}|\leq s_{0j}$ for all $j$. Let $\delta(b_j)=k_j(=p_{0j})$ be the degree of $b_j$ for all $j$ such that $k_0=k_i=1$ for some $i$.
\begin{enumerate}
  \item Then $|p_{ij}|= k_j$ for all $j$.
  \item If $s$ is a real matrix then $p_{ij}= \pm k_j$ for all $j$.
\end{enumerate}

\end{lemma}

\begin{proof}
For $(i)$. Since $\delta(b_i)=1$,  $d_i= d_0$. The first and $i$th rows of $s$ are $[1,1,\sqrt{k_2},\sqrt{k_3},\hdots, \sqrt{k_{r-1}}]$ and $[1,p_{i1},p_{i2}/\sqrt{k_2},p_{i3}/\sqrt{k_3},\hdots, p_{i,r-1}/\sqrt{k_{r-1}}]$, respectively. Since $d_0=d_i$, $|p_{ij}|/\sqrt{k_j}= \sqrt{k_j}$ for all $j$. Hence $|p_{ij}|= k_j$ for all $j$.

\smallskip

For $(ii)$. By Part $(i)$, $|p_{ij}|= k_j$ for all $j$.   Since $s$ is a real matrix, $p_{ij}= \pm k_j$ for all $j$.
\end{proof}

\begin{lemma}\label{EqualityOfKandLLemmaRank4}Let $(A,\BB, \delta)$ be a $C$-algebra arising from a real Allen matrix $s$ of rank $4$  with standard  basis $\BB=\{b_0,b_1,\hdots, b_3\}$. Let $|s_{ij}|\leq s_{0j}$ for all $j$. If $\delta(b_1)=1$, $\delta(b_2)=k_2$ and $\delta(b_3) =k_3$, then either $k_3=k_2$ or $k_3=k_2+2$.
\end{lemma}

\begin{proof}By Lemma \ref{EqualityOfKandLLemmaRank5},
 $p_{11}=\pm 1, p_{12} = \pm k_2$ and $p_{13}=\pm k_3$. Therefore, if $p_{11} = -1$, $k_3=k_2$, and if $p_{11} = 1$, then $k_3=k_2+2$.
\end{proof}



The following proposition classify the non-homogeneous Allen matrices of rank $4$ with one nontrivial degree equal to $1$.

\begin{prop}\label{CharacterTableExpreesionWithDifferentNormPropRank4} Let $(A,\BB, \delta)$ be a non-homogeneous $C$-algebra arising from an Allen matrix $s$ of rank $4$ with standard  basis $\BB=\{b_0,b_1,\hdots, b_3\}$. Let $|s_{ij}|\leq s_{0j} $ for all $j$.  If $\delta(b_i)=1$ for one $i>0$ and $\delta(b_j)=k_j$ for all $j \neq i$.  Then the character table $P$ has the following expressions up to simultaneous permutation of rows and columns. (Though we remark that the simultaneous permutations of the  row  and column  of their corresponding Allen matrices  gives us again  Allen matrices.)
$$\left[
   \begin{array}{cccc}
     1 & 1 & 2 &2\\
     1 & -1 &    2 &  -2\\
       1 &  1 &  -1&  -1\\
         1 & -1 &  -1&  1\\
   \end{array}\right], ~~\left[
   \begin{array}{cccc}
     1 & 1 & 2 &4\\
     1 &  1 &    2 &  -4\\
       1 &  1 &  -2&  0\\
         1 & -1 &  0& 0\\
   \end{array}\right]\mbox{ and } \left[
   \begin{array}{cccc}
     1 & 1 & 4 &6\\
     1 &  1 &   4 &  -6\\
       1 &  1 &  -2&  0\\
         1 & -1 &  0& 0\\
   \end{array}\right].$$
\end{prop}

\begin{proof} Without loss of generality, let $k_1=1$, see Theorem \ref{SimulataneousPermThm}. Thus the character table
$$P=\left[
   \begin{array}{cccc}
     1 & 1 & k_2 & k_3\\
     1 & p_{11} &   p_{12} & p_{13}\\
       1 & p_{21} &  p_{22}&  p_{23}\\
         1 & p_{31} &  p_{32}&  p_{33}\\
   \end{array}
 \right].$$

Case 1. If $p_{11},  p_{12}$ and  $p_{13}$ are not rational integers.

\smallskip

Since $d_0=d_1$, $|p_{11}|=1, |p_{12}|=k_2$ and $|p_{13}|=k_3$. Thus $p_{11},  p_{12}$ and  $p_{13}$ cannot be irrational, hence they can be non-real algebraic integers. As complex conjugate of an irreducible character is an irreducible character, without loss of generality we assume the third irreducible character is a complex conjugate of the second character, thus we have $k_2=1$, see Theorem \ref{SimulataneousPermThm}. Therefore, the entries of the fourth row of $P$ might be rational integers and as $A$ is non-homogeneous, by Proposition \ref{SymmetrizingPropertyProp}, $d_0=d_im_i$  implies $k_3=3$ and $d_3=2$. Hence there might be two zeros in the fourth row. But each entry of the rows $1, 2$ and $3$ is non-zero. As $S$ is symmetric, we get a contradiction.

\medskip

Case 2. If $p_{11},  p_{12}$ and  $p_{13}$ are rational integers.

\smallskip

By Lemma \ref{EqualityOfKandLLemmaRank4}, the second row of $P$ is either $[1,-1,k_2,-k_2]$ or $[1,1,k_2,-(k_2+2)]$.

\medskip

Subcase 1. Let the second row of $P$ be $[1,-1,k_2,-k_2]$.

\smallskip

Then, the first line of the character table is $[1,1,k_2,k_2]$.
Using symmetry of the matrix $S$ and orthogonality of characters, we have
$$P=\left[
   \begin{array}{cccc}
     1 & 1 & k_2 &k_2\\
     1 & -1 &    k_2 &  -k_2\\
      1 &  1 &  -1&  -1\\
         1 & -1 &  -1&  1\\
   \end{array}
 \right].$$
Since $A$ is non-homogenous, $k_2\neq1$. As $k_2$ divides $d_0$, therefore $k_2=2$. Hence
$$P=\left[
   \begin{array}{cccc}
     1 & 1 & 2 &2\\
     1 & -1 &    2 &  -2\\
       1 &  1 &  -1&  -1\\
         1 & -1 &  -1&  1\\
   \end{array}\right]$$

\medskip

Subcase 2. Let the second row of $P$ be $[1, 1,k_2,-(k_2+2)]$.

\smallskip

Then, the first line of the character table is $[1,1,k_2,k_2+2]$.
The row sum of the character table is zero for each row, using symmetry of the matrix $S$ and orthogonality of characters, we have
$$P=\left[
   \begin{array}{cccc}
     1 & 1 & k_2 &k_2+2\\
     1 & 1 &    k_2 &  -(k_2+2)\\
     1 & 1 &  -2&  0\\
     1 &-1 &  0&  0\\
   \end{array}
 \right].$$
Since $A$ is non-homogenous, $k_2\neq1$. As $k_2$ divides $d_0$, therefore $k_2=2$ or $4$. Hence
$$ P=\left[
   \begin{array}{cccc}
     1 & 1 & 2 &4\\
     1 &  1 &    2 &  -4\\
       1 &  1 &  -2&  0\\
         1 & -1 &  0& 0\\
   \end{array}\right]\mbox{ or } P=\left[
   \begin{array}{cccc}
     1 & 1 & 4 &6\\
     1 &  1 &   4 &  -6\\
       1 &  1 &  -2&  0\\
         1 & -1 &  0& 0\\
   \end{array}\right].$$

\medskip

Subcase 3. Let the second row of $P$ be $[1, 1,-k_2,k_2-2)]$.

\smallskip

Then, the first line of the character table is $[1,1,k_2,k_2-2]$.
The row sum of the character table is zero for each row, using symmetry of the matrix $S$ and orthogonality of characters, we have
$$P=\left[
   \begin{array}{cccc}
      1 & 1 & k_2 &k_2-2\\
     1 & 1 &    -k_2 &  k_2-2\\
      1 & -1 &  p_{22}&  0\\
         1 & 1 & 0&  -2\\
   \end{array}
 \right].$$
Since $(k_2-2)\big|(k_2+2)$, $k_2\in\{3,4,6\}$.  As $|s_{ij}|\leq s_{0j}$ therefore $k_2=3$ is not possible. For $k_2\in\{4,6\}$, $d=2$ implies $p_{22}=0$. Thus the following are the only possible values of $P$.
$$\left[
   \begin{array}{cccc}
      1 & 1 & 4 &2\\
     1 & 1 &    -4 &  2\\
      1 & -1 &  0&  0\\
         1 & 1 & 0&  -2\\
   \end{array}
 \right],~~\left[
   \begin{array}{cccc}
      1 & 1 & 6 &4\\
     1 & 1 &    -6 &  4\\
      1 & -1 &  0&  0\\
         1 & 1 & 0&  -2\\
   \end{array}
 \right].$$
Note that corresponding to each of the above $P$ matrix the Allen matrix $s$ has nonnegative integral structure constants $N_{ijk}$. Since the above matrices are real matrices, any permutation of $i,j,k$ in the subscript of $N_{ijk}$ gives us the same structure constants. Hence,  if we use the fact that the above matrices being the character tables of the association schemes have the (nonnegative) integral structure constants then  we can easily verify the integrality of the structure constants $N_{ijk}$ through the Allen integrality condition, see Definition \ref{AllenIntDef} and Theorem \ref{IntegralCAlgebrasAndFourierMatrices}.
\end{proof}

\begin{remark}\rm We note that all the Allen matrices in the above proposition are not integral Fourier matrices because the degrees are not square integers, see Lemma \ref{SqaureDegreesForIntegralFurierCaseLemma}. In the above case the first character table is exactly same as the character table for the association schemes {\tt as6(5)}  up to simultaneous permutation of rows $3,4$. The remaining two are the character tables of the association schemes {\tt as8(4)} and {\tt as12(8)}.
\end{remark}

Since the row sum for the character table of a $C$-algebra is zero, if $\delta(b_i)=1$ we can determine the $i$th row of the character table $P$. In the next proposition we assume $r-1$ entries of the second row and adjust the $r$th entry of the second row using the fact that row sum is zero and $d_i=d_0$ then determine whether there is an Allen matrix of rank $5$.

%
%
%
%

\begin{prop}\label{CharacterTableExpreesionWithDifferentNormPropRank5}
Let $(A,\BB, \delta)$ be a non-homogeneous $C$-algebra arising from an Allen matrix $s$ of rank $5$ with standard  basis $\BB=\{b_0,b_1,\hdots, b_4\}$. If $|s_{ij}|\leq s_{0j} $ for all $j$.  If $\delta(b_i)=1$ for one $i>0$ and $\delta(b_j)=k_j$ for all $j \neq i$. Then the character table has the following expression up to simultaneous permutation of row and columns.
$$\left[
   \begin{array}{ccccc}
     1 & 1 & 2 &2&2\\
     1 & 1 &  2 &  -2&-2\\
      1 & 1 & -2& 0&0\\
     1 &  -1 & 0&   \sqrt{2} &-\sqrt{2}\\
     1 & -1 & 0&  -\sqrt{2} & \sqrt{2}\\
   \end{array}\right], ~~\left[
   \begin{array}{ccccc}
     1 & 1 & 2 &4&8\\
     1 & 1 &   2 & 4&-8\\
      1 &  1 & 2&  -4& 0\\
     1 & 1 &  -2&  0& 0\\
     1 &  -1 &  0&  0& 0\\
\end{array}\right],
   \left[
   \begin{array}{ccccc}
     1 & 1 & 4 &3&3\\
     1 & 1 &  4 &  -3&-3\\
      1 & 1 & -2& 0&0\\
     1 &  -1 & 0&   \sqrt{3} &-\sqrt{3}\\
     1 & -1 & 0&  -\sqrt{3} & \sqrt{3}\\
   \end{array}\right].
$$
\end{prop}

\begin{proof}

Without loss of generality, let $k_1=1$, see Theorem \ref{SimulataneousPermThm}. Thus the character table
$$P=\left[
   \begin{array}{ccccc}
     1 & 1 & k_2 & k_3&k_4\\
     1 & p_{11} &   p_{12} & p_{13}&p_{14}\\
     1 &  p_{21} &  p_{22}&  p_{23}& -(1+p_{21} +  p_{22}+  p_{23})\\
     1 &  p_{31} &  p_{32}&  p_{33}& -(1+p_{31} +  p_{32}+  p_{33})\\
     1 &  p_{41} &  p_{42}&  p_{43}& -(1+p_{41} +  p_{42}+  p_{43})\\
   \end{array}
 \right].$$

Case 1. If $p_{11},  p_{12}, p_{13}$ and  $p_{14}$ are not all rational integers.

\smallskip

Since $d_0=d_1$, $|p_{11}|=1, |p_{12}|=k_2, |p_{13}|=k_3$ and $|p_{14}|=k_4$. Since the row sum is zero, at least two of these $p_{11},  p_{12}, p_{13}$ and  $p_{14}$ can be non-real algebraic integers. As complex conjugate of an irreducible character is an irreducible character, without loss of generality, we assume that the third irreducible character is a complex conjugate of the second character, thus we have $k_2=1$. Without loss of generality, let $k_3\leq k_4$.  By Proposition \ref{SymmetrizingPropertyProp}, $d_0=d_im_i$  implies $(k_3,k_4) \in \{(1, 2),
(1, 4),
(2, 5),
(3, 6),
(6, 9)
\}$.

If $k_3=1$ then $d_3=d_0=d_1=d_2$ implies all the entries of the row $1,2,3$ and $4$ are non zero. Since $k_3\neq k_4$, the entries of the fifth row  are rational integers as the complex conjugate of an irreducible character is an irreducible character. Thus there might be at least two zero entries in the fifth row of $P$. But the matrix $S$ is symmetric. Hence we get a contradiction.

For $(k_3,k_4) \in \{(2, 5),
(3, 6),
(6, 9)
\}$ as $k_3\neq k_4$, the entries of the row $5$ are rational integers because $k_2=k_1=k_0=1$ and the complex conjugate of an irreducible character is an irreducible character. Each entry of the row $1, 2$ and $3$ of $P$ is non-zero. But for each the above pair there are exactly $3$ zeros in the fifth row. Since $S$ is a symmetric matrix, we get a contradiction.

\medskip

Case 2. If $p_{11},  p_{12}, p_{13}$ and  $p_{14}$ are all rational integers.

\smallskip

By Lemma \ref{EqualityOfKandLLemmaRank5} and Theorem \ref{SimulataneousPermThm} the only possible degree patterns are:

$[1, -1, k_2, k_3, -(k_2+k_3)]$, $[1, 1, k_2, k_3, -(k_2+k_3+2)]$, $[1, ~ 1,~     k_2,   ~-k_3,~ -(k_2-k_3+2)]$,

$[1, ~ -1,~     k_2,   ~-k_3,~ -(k_2-k_3)]$, $[1, ~ 1,~     -k_2,   ~-k_3,~ k_2+k_3-2]$, $[1, ~ -1,~     -k_2,   ~-k_3,~ k_2+k_3]$.

\medskip

Subcase 1. Let the second row of $P$ be $[1, -1, k_2, k_3, -(k_2+k_3)]$.

\smallskip

Then the first row of the character table is $[1, 1, k_2, k_3, k_2+k_3]$ and $(k_2+k_3)\big|(2+k_2+k_3)$. Thus $(k_2,k_3)=(1,1)$.
Therefore, by the orthogonality of characters, we have
$$P=\left[
   \begin{array}{ccccc}
     1 & 1 & 1 &1&2\\
     1 & -1 &    1 &  1&-2\\
     1 &  1 &  p_{22}&  p_{23}& -1\\
     1 & 1 &  p_{32}&  p_{33}& -1\\
     1 &  -1 & p_{42}&  p_{43}& 1\\
   \end{array}\right].
$$
As $m_2=m_3=1$, therefore $d_2=d_3=6$. But each of $|p_{22}|$, $|p_{23}|$, $|p_{32}|$ and $|p_{33}|$ can be at most $1$, thus both $d_2$ and $d_3$ are strictly less than  $6$, a contradiction. Hence this case is not possible.

\medskip

Subcase 2. Let the second row of $P$ be $[1, 1, k_2, k_3, -(k_2+k_3+2)]$.

\smallskip

Then the first row of the character table is $[1, 1, k_2, k_3, k_2+k_3+2]$.
Therefore, by the orthogonality of the irreducible characters and symmetry of the matrix $S$, we have
$$P=\left[
   \begin{array}{ccccc}
     1 & 1 & k_2 &k_3&k_2+k_3+2\\
     1 & 1 &   k_2 &  k_3&-(k_2+k_3+2)\\
      1 &  1 &  p_{22}&  -2-p_{22}& 0\\
     1 & 1 &  p_{32}&  -2-p_{32}& 0\\
     1 &  -1 &  0&  0& 0\\
   \end{array}\right].
$$
Without loss of generality, let $k_2\leq k_3$. Since $k_2\big|(2k_3+4)$ and $k_3\big|2k_2+4$, we have

$(k_2,k_3)\in \{(1, 2),
(1, 3),
(1, 6),
(2, 4),
(2, 8),
(3, 10),
(4, 6),
(4, 12),
(6, 16),
(8, 10),
(12, 28)\}
$.
%
%

Note that $k_2\neq k_3$, thus $d_3\neq d_4$. But the  conjugate irreducible characters have same norm. Thus $p_{22},  p_{23}, p_{32}$ and $p_{33}$
are rational integers. Therefore,   det$(P) \in\ZZ$ and $(\mbox{det}P)^2=n$. Thus $n=2(k_2+k_3+2)$ need to be a square. But the only two pairs $(2, 4),
(4, 12)$ do not fail this test.

For $(k_2,k_3)=(2, 4)$, $d_2=8=1+1+\big(\dfrac{p_{22}}{\sqrt2}\big)^2+\big(\dfrac{-2-p_{22}}{\sqrt4}\big)^2$. Since the entries of $P$ are algebraic integers, we have $p_{22}=2$.
Similarly, $d_3=4$ implies $p_{32}=-2$.
Thus the character table
$$P=\left[
   \begin{array}{ccccc}
     1 & 1 & 2 &4&8\\
     1 & 1 &   2 & 4&-8\\
      1 &  1 & 2&  -4& 0\\
     1 & 1 &  -2&  0& 0\\
     1 &  -1 &  0&  0& 0\\
\end{array}\right].$$

For $(k_2,k_3)=(4, 12)$, $d_2=9$ and  $d_3=3$. Therefore, we have  $p_{22}=4,-5$ and  $p_{32}=1, -2$.

\smallskip

Thus the possible character tables are:
$$P_1=\left[
   \begin{array}{ccccc}
     1 & 1 & 4 &12&18\\
     1 & 1 &   4 & 12&-18\\
      1 &  1 & 4&  -6& 0\\
     1 & 1 &   1&  -3& 0\\
     1 &  -1 &  0&  0& 0\\
   \end{array}\right],\mbox{  } P_2=\left[
   \begin{array}{ccccc}
     1 & 1 & 4 &12&18\\
     1 & 1 &   4 & 12&-18\\
      1 &  1 & -5&  3& 0\\
     1 & 1 &   1&  -3& 0\\
     1 &  -1 &  0&  0& 0\\
   \end{array}\right],$$
$$P_3=\left[
   \begin{array}{ccccc}
     1 & 1 & 4 &12&18\\
     1 & 1 &   4 & 12&-18\\
      1 &  1 & 4&  -6& 0\\
     1 & 1 &  -2&  0& 0\\
     1 &  -1 &  0&  0& 0\\
   \end{array}\right],\mbox{   } P_4=\left[
   \begin{array}{ccccc}
     1 & 1 & 4 &12&18\\
     1 & 1 &   4 & 12&-18\\
      1 &  1 & -5&  3& 0\\
      1 & 1 &  -2&  0& 0\\
     1 &  -1 &  0&  0& 0\\
   \end{array}\right].$$
Since an Allen matrix have algebraic integer entries, the above matrices are not the character tables of the $C$-algebras arising from Allen matrices, see \cite[Proposition 2.17]{HIB2}, \cite{MC1}.

\medskip

Subcase 3. Let the second row of $P$ be $[1, ~ 1,~     k_2,   ~-k_3,~ -(k_2-k_3+2)]$.

\smallskip

Then, the first row of the character table is $[1, 1, k_2, k_3, k_2-k_3+2]$.
Therefore, by the orthogonality of characters, symmetry of the matrix $S$ and $P\bar P=nI$, we have
$$P=\left[
   \begin{array}{ccccc}
     1 & 1 & k_2 &k_3&k_2-k_3+2\\
     1 & 1 &   k_2 &  -k_3&-(k_2-k_3+2)\\
      1 & 1 & -2& 0&0\\
     1 &  -1 & 0&  p_{33}& -p_{33}\\
     1 & -1 & 0&  p_{43}& -p_{43}\\
   \end{array}\right].
$$
Therefore $k_2\big|4$, $k_3\big|2k_2+4$,  $(k_2-k_3+2)\big|(k_2+k_3+2)$ and $k_2-k_3+2>0$.
Hence $(k_2,k_3)\in \{(1, 1),
(1, 2),
(2, 2),
(4, 2),
(4, 3),
(4, 4)\}$.
Since $|s_{ij}|\leq s_{0j}$, $(k_2,k_3) \not \in \{(1,1),(1,2)\}$.
%
For $(k_2,k_3)=(2,2)$, $d_3=d_4=4$. Thus $p_{33}\bar p_{33}=2$, $p_{43}\bar p_{43}=2$. But the integrality of the structure constants and orthogonality of characters forces $p_{33}=  \pm\sqrt{2}$  and $p_{43}=  \mp\sqrt{2}$. Therefore, up to simultaneous permutation of row $4$ and row $5$, and column $4$ and column $5$, we have
$$P=\left[
   \begin{array}{ccccc}
     1 & 1 & 2 &2&2\\
     1 & 1 &  2 &  -2&-2\\
      1 & 1 & -2& 0&0\\
     1 &  -1 & 0&   \sqrt{2} &-\sqrt{2}\\
     1 & -1 & 0&  -\sqrt{2} & \sqrt{2}\\
   \end{array}\right].
$$
\medskip

For $(k_2,k_3)=(4,2)$, $d_3=6$
implies $|p_{33}|=\dfrac{4}{\sqrt3}>2$, a contradiction.



For $(k_2,k_3)=(4,3)$, $k_4=3$, $d_3=4$ and $d_4=4$
implies $|p_{33}|= \sqrt3$ and $|p_{43}|= \sqrt3$. But the integrality of structure constants and orthogonality of characters forces $p_{33}=  \pm\sqrt{3}$  and $p_{43}=  \mp\sqrt{3}$. Therefore, up to simultaneous permutation of row $4$ and row $5$, and column $4$ and column $5$, we have
$$P=\left[
   \begin{array}{ccccc}
     1 & 1 & 4 &3&3\\
     1 & 1 &  4 &  -3&-3\\
      1 & 1 & -2& 0&0\\
    1 &  -1 & 0&   \sqrt{3} &-\sqrt{3}\\
     1 & -1 & 0&  -\sqrt{3} & \sqrt{3}\\
   \end{array}\right].
$$
Although the structure constants are not all integers, for example $\lambda_{342}=3/2$, but it satisfies the Allen integrality condition, see Definition \ref{AllenIntDef}. Hence the $P$ is a character table of a $C$-algebra arising from an Allen matrix.

\medskip

Subcase 4. Let the second row of $P$ be $[1, ~ -1,~     k_2,   ~-k_3,~ -(k_2-k_3)]$.

\smallskip

Then, the first row of the character table is $[1,~ 1, ~k_2, ~k_3, ~ k_2-k_3]$.
Therefore, by the orthogonality of  the characters and symmetry of the matrix $S$, we have
$$P=\left[
   \begin{array}{ccccc}
     1 & 1 & k_2 &k_3&k_2-k_3\\
     1 & -1 &    k_2 &  -k_3&-(k_2-k_3)\\
     1 &  1 &  0&  p_{23}& -(2+ p_{23})\\
     1 & -1 &  p_{32}&  p_{33}& -(p_{32}+  p_{33})\\
     1 & -1 &  p_{42}&  p_{43}& -(p_{42}+  p_{43})\\
   \end{array}\right].
$$
Since $P\bar P=nI$, from row $1$, $2$ and column $3$, we get $k_2=0$, a contradiction. Hence this case is not possible.

\medskip

Subcase 5. Let the second row of $P$ be $[1, ~ 1,~     -k_2,   ~-k_3,~ k_2+k_3-2]$.

\smallskip

Then the first row of the character table is $[1, ~1, ~k_2, ~k_3, k_2+k_3-2]$.
Therefore, by the symmetry of the matrix $S$ and $P\bar P=nI$, we have
$$P=\left[
   \begin{array}{ccccc}
     1 & 1 & k_2 &k_3&k_2+k_3-2\\
     1 & 1 &   -k_2 &  -k_3& k_2+k_3-2\\
      1 &  -1  &  p_{22}&  p_{23}& -(p_{22}+  p_{23})\\
     1 &  -1 &  p_{32}&  p_{33}& -(p_{32}+  p_{33})\\
     1 &  1 &  p_{42}&  p_{43}& -(2 +  p_{42}+  p_{43})\\
   \end{array}\right].
$$
Now from orthogonality of characters and symmetry of the matrix $S$, we have
$$P=\left[
   \begin{array}{ccccc}
     1 & 1 & k_2 &k_3&k_2+k_3-2\\
     1 & 1 &   -k_2 &  -k_3& k_2+k_3-2\\
      1 & -1 &  p_{22}&  -p_{22}&0\\
     1 & -1 &  p_{32}&  -p_{32}& 0\\
     1 &  1 &  0&  0& -2\\
   \end{array}\right].
$$
Therefore $k_2\big|2k_3$, $k_3\big|2k_2$,  $(k_2+k_3-2)\big|(k_2+k_3+2)$ and $k_2+k_3-2>0$.
Without loss of generality, let $k_2\leq k_3$. Hence $(k_2,k_3)\in \{(1, 2),
(2, 2)\}$.
Since $|s_{ij}|\leq s_{0j}$, $(k_2,k_3)\neq (1,2)$.

\medskip

For $(k_2,k_3)=(2,2)$, $d_2=d_3=4$. Thus $p_{22}\bar p_{22}=2$, $p_{32}\bar p_{32}=2$. But the integrality of structure constants and orthogonality of characters forces $p_{22}=  \pm\sqrt{2}$  and $p_{32}=  \mp\sqrt{2}$.
Therefore, up to simultaneous permutation of rows and columns, 
we have
$$P=\left[
   \begin{array}{ccccc}
     1 & 1 & 2 &2&2\\
     1 & 1 &  -2 &  -2&2\\
       1 & -1 &  \sqrt{2} &-\sqrt{2}&0\\
     1 & -1 & -\sqrt{2} & \sqrt{2}& 0\\
     1 &  1 &  0&  0& -2\\
   \end{array}\right].
$$

\medskip

Subcase 6. Let the second row of $P$ be $[1, ~ -1,~     -k_2,   ~-k_3,~ k_2+k_3]$.

\smallskip

Then the first row of the character table is $[1,~ 1,~ k_2,~ k_3, ~k_2+k_3]$.
Therefore, by the symmetry of the matrix $S$ and orthogonality of the characters, we have
$$P=\left[
   \begin{array}{ccccc}
     1 & 1 & k_2 &k_3&k_2+k_3\\
     1 & -1 &    -k_2 &  -k_3&k_2+k_3\\
      1 &  -1 &  p_{22}&  p_{23}& -1\\
     1 &  -1 &  p_{32}&  p_{33}& -1\\
     1 &  1 &  p_{42}&  p_{43}& 0\\
   \end{array}\right].
$$
Similar to the Subcase 1, $k_2=k_3=1$ implies $d_2=d_3=6$, and we get a contradiction.
\end{proof}

\begin{remark}\rm
We note that all the above Allen matrices in the above proposition are not integral Fourier matrices as the multiplicities are not square integers. In the above proposition, the first two matrices are the character table for {\tt as08(10), as16(24)}, respectively. The third matrix of order $12$ is not the character table of  an association scheme as it has negative structure constants.
\end{remark}

\section{$C$-algebras arising from integral Fourier matrices of rank $4$ and $5$}

In the previous section we point out that there is no integral Fourier matrix of rank $3$ and there is just one integral Fourier matrix of rank $2$. In this section we prove that for rank $4$ and $5$ there is no integral Fourier matrix, but unlike the previous section we do not assume any nontrivial degree to be equal to $1$.

\begin{thm}\label{FourDifferentDegreesThm4}Let $(A,\BB, \delta)$ be a non-homogenous $C$-algebra arising from an Allen matrix $s$  of rank $4$ with standard  basis $\BB=\{b_0,b_1,\hdots, b_3\}$. Then the Allen matrix cannot be integral Fourier matrix. In other words, there is no $C$-algebra of rank $4$ arising from an integral Fourier matrix.
\end{thm}

\begin{proof} Suppose the Allen matrix  is an integral Fourier matrix. Let $\delta(b_i)= k_i$, for all $i>0$. Since $s$ is an integral Fourier matrix, $k_1, k_2$ and $k_3$ are square integers, see Lemma \ref{SqaureDegreesForIntegralFurierCaseLemma}. As $d_0=1+k_1+k_2+k_3$ and
$k_1, k_2$ and $k_3$ divide $d_0$. Therefore, $k_2+k_3\equiv-1\mod k_1$, $k_1+k_3\equiv-1\mod k_2$ and $k_1+k_2\equiv-1\mod k_3$.

Claim: $k_1=k_2=k_3=1$. Without loss of generality, suppose $k_1$ is an even integer.  Since $k_1, k_2$ and $k_3$ are squares, we have $k_1\equiv 0 \mod4$ and $d_0\not\equiv 0 \mod4$, a contradiction to fact that $k_1$ divides $d_0$. Therefore $k_1, k_2$ and $k_3$ are odd integer. Suppose $k_1\geq k_2, k_3$ and $k_1>1$.
Now, if all $k_1,k_2$ and $k_3$ are odd integers then
$k_1,k_2,k_3\equiv 1 \mod4$. But $d_0\equiv 0\mod4$  implies $d_0=k_1a$, $a\geq 4$. Therefore, $k_1(a-1)=1+k_2+k_3$ implies $3k_1\leq 1+k_2+k_3$. Thus $k_2$ or $k_3>k_1$, again a contradiction.
\end{proof}

\begin{thm}\label{FourDifferentDegreesThm5}Let $(A,\BB, \delta)$ be a non-homogenous $C$-algebra arising from an Allen matrix $s$  of rank $5$ with standard  basis $\BB=\{b_0,b_1,\hdots, b_4\}$. Then the Allen matrix cannot be integral Fourier matrix. In other words, there is no $C$-algebra of rank $5$ arising from an integral Fourier matrix.
\end{thm}

\begin{proof}Suppose the Allen matrix is an integral Fourier matrix. Let $\delta(b_i)=k_i$, for all $i>0$. Since $s$ is an integral Fourier matrix, $k_1,k_2,k_3$ and $k_4$ are squares and integers, see Lemma \ref{SqaureDegreesForIntegralFurierCaseLemma}. By \cite[Lemma 3.7]{MC1}, $d_0$ is a square. Therefore, $d_0 \equiv 0,1 \mod4$.
Let $d_0=k_4a, d_0=k_3b,d_0=k_2c,d_0=k_1d$. Then each of $a,b,c$ and $d$ is greater than $1$ and a square integer as $A$ is non-homogeneous and $d_0$ is a square.

\smallskip

Case 1. Each of $k_1,k_2,k_3,k_4$ is an odd integer.

Then $d_0$ is odd, $k_1,k_2,k_3,k_4$ are odd implies $a,b,c,d$ are odd and greater than or equal to $9$. Without loss of generality, let $k_4\geq k_1,k_2,k_3$. Therefore, $d_0\geq 9k_4, d_0=1+k_1+k_2+k_3+k_4\leq 1+4k_4$, a contradiction.

\smallskip

Case 2. Three of $k_1,k_2,k_3,k_4$ are  odd and one is even.

Then $d_0$ is even. Without loss of generality, let $k_4$ is even. Thus $d_0=k_4a$ implies $a\geq 4$.

\smallskip

Subcase 1. If $k_4>k_1,k_2,k_3$.

Then $d_0\geq 4k_4$, $d_0\leq 1+(k_4-1)+(k_4-1)+(k_4-1)+k_4=4k_4-2$, a contradiction.

\smallskip

Subcase 2. If $k_4<$ one of $k_1,k_2,k_3$, say $k_3$, so $k_1,k_2\leq k_3$.

Then $d_0\geq 4k_3$ as $b$ is an even square. Thus $d_0\leq 1+k_3+k_3+k_3+(k_3-1)=4k_3$ implies $k_1=k_2=k_3$ and $k_4=k_3-1$, $d_0=4k_3$. Now $d_0=4x^2$ because $d_0$ is a square and an even integer. Hence $k_1=k_2=k_3=x^2$ and $k_4=x^2-1$. As $x^2-1$ divides $4x^2$ and $x$ is an odd integer thus  $x^2-1$ divides $4$, we get a contradiction.

\smallskip

Case 3. Two of $k_1,k_2,k_3,k_4$ are  odd and two are even.

Then $d_0\equiv 3 \mod 4$, a contradiction. 

\smallskip

Case 4. One of $k_1,k_2,k_3,k_4$ is odd and three are even.

Then $d_0\equiv 2 \mod 4$, a contradiction. 
\end{proof}

\end{document}